\documentclass[a4paper,leqno,10pt]{amsart}

\usepackage{epsf,graphicx,epsfig}
\usepackage{amscd}
\usepackage{amsmath,latexsym,amssymb,amsthm}
\usepackage[nospace,noadjust]{cite}
\usepackage{setspace,cite}
\usepackage{lscape,fancyhdr,fancybox}
\usepackage{graphicx}

\usepackage{color}
\usepackage{url}
\usepackage{enumerate}
\usepackage[mathscr]{euscript}

  \theoremstyle{plain}

\swapnumbers
    \newtheorem{thm}{Theorem}[section]
    \newtheorem{prop}[thm]{Proposition}
   \newtheorem{lemma}[thm]{Lemma}

    \newtheorem{subsec}[thm]{}
\theoremstyle{definition}
    \newtheorem{defn}[thm]{Definition}
    \newtheorem{exam}[thm]{Example}

\theoremstyle{remark}
     \newtheorem{remark}[thm]{Remark}

\title{}
\author{}
\date{}
\usepackage{amssymb}

\begin{document}
\title{Cup-product for equivariant Leibniz cohomology and zinbiel algebras}

\author{Goutam Mukherjee}
\email{gmukherjee.isi@gmail.com}
\address{Stat-Math Unit,
 Indian Statistical Institute, Kolkata 700108,
West Bengal, India.}

\author{Ripan Saha}
\email{ripanjumaths@gmail.com}
\address{Department of Mathematics,
Raiganj University, Raiganj, 733134,
West Bengal, India.}

\date{\today}
\subjclass[2010]{$17$A$30$, $17$A$32$, $17$B$56$, $55$N$91$.}

\thispagestyle{empty}

\begin{abstract}
We study finite group actions on Leibniz algebras, define equivariant cohomology groups associated to such actions. We show that there exists a cup-product operation on this graded cohomology groups which makes it a graded zinbiel algebra.
\end{abstract}
\maketitle

\section{Introduction}
In \cite{L4}, J.-L. Loday introduced some new types of algebras along with their (co)homologies and studied the associated operads. Leibniz algebras and their Koszul duals, zinbiel algebras are examples of such algebras.  Let {\bf Leib} be the category of Leibniz algebras over a fixed field $\mathbb{K}.$ Given any Leibniz algebra $\mathfrak g,$ J.-L. Loday \cite{L3} introduced a cup-product operations 
$$\cup:HL^p(\mathfrak g; A) \times HL^q(\mathfrak g; A)\rightarrow HL^{p+q}(\mathfrak g; A)$$
on the graded Leibniz cohomology $HL^\ast (\mathfrak g; A)$ groups with coefficients in a commutative, associative algebra $A.$ This product is neither associative nor commutative, but satisfies the formula
$$([a]\cup [b]) \cup [c] = [a]\cup ([b]\cup [c]) + (-1)^{|[b]||[c]|} [a]\cup ([c]\cup [b]),$$ which is, the defining relation of a zinbiel algebra. Thus, the author proved that $HL^\ast(\mathfrak g; A)$ is a graded zinbiel algebra. The aim of this paper is to study finite group actions on Leibniz algebras. Let $G$ be a finite group and $\mathfrak g$ be a Leibniz algebra equipped with a given action of $G.$ We discuss examples of such actions and introduce equivariant cohomology groups of a Leibniz algebra $\mathfrak g$ equipped with an action of a finite group $G,$ along the line of Bredon cohomology of a $G$-space \cite{bredon67}. We introduce a cup-product operation in the equivariant context and prove that for a Leibniz algebra $\mathfrak g$ equipped with an action of $G$, equivariant graded Leibniz cohomology groups also admit a graded zinbiel algebra structure.    

\section{Preliminaries}\label{prelim} 
In this section, we recall some definitions, notations and results from \cite{L2},\cite{L3}, \cite{L4}. 

In \cite{L0}, J.-L. Loday observed that for a Lie algebra $\mathfrak g$ if one replaces the exterior product $\wedge$ by the tensor product $\otimes$ in the classical formula for the boundary map $d$ of the Chevalley-Eilenberg complex and modifies the boundary map $d$ so as to put the commutator $[x_i, x_j]$ at the place $i$ when $i < j$ (see 10.6.2.1, \cite{L0}), then one obtains a new complex $(T\mathfrak g, d).$ The only relation that is used to get $d^2= 0$ is
$$[x,[y,z]]= [[x,y],z]-[[x,z],y] ~~\mbox{for}~x,~y,~z \in \mathfrak g.$$
Dualizing this complex one gets the Leibniz cohomology of the Lie algebra $\mathfrak g.$ Thus, Leibniz cohomology is defined for a larger class of algebras: the Leibniz algebras. More explicitly, it is defined as follows.

\begin{defn}   
Let $\mathbb{K}$ be a field. A Leibniz algebra is a vector space $\mathfrak g$ over $\mathbb{K},$ equipped with a bracket operation, which is $\mathbb{K}$-bilinear and satisfies the Leibniz identity: 
$$[x,[y,z]]= [[x,y],z]-[[x,z],y] ~~\mbox{for}~x,~y,~z \in \mathfrak g.$$
A graded Leibniz algebra is a graded $\mathbb{K}$-vector space $\mathfrak g = \{\mathfrak g_i\}_{i\geq 0},$ equipped with a graded bracket operation of degree $0$ which is $\mathbb{K}$-bilinear and satisfies the graded Leibniz identity: $[x,[y,z]]= [[x,y],z]-(-1)^{|y||z|}[[x,z],y]$ for homogeneous elements  $x,~y,~z \in \mathfrak g.$
\end{defn}
Any Lie algebra is automatically a Leibniz algebra, as in the presence of skew symmetry, the Jacobi identity is equivalent to the Leibniz identity. 
\begin{exam}\label{example-1}
Let $(\mathfrak g,d)$ be a differential Lie algebra with the Lie bracket $[,]$. Then $\mathfrak g$ is a Leibniz algebra with the bracket operation $[x,y]_d:= [x,dy]$. The new bracket on $\mathfrak g$ is  called the derived bracket.
\end{exam}
\begin{exam}\label{ex}\label{example-2}
Consider a three dimensional vector space $\mathfrak g$ spanned by \linebreak$\{e_1,~e_2,~e_3\}$ over $\mathbb{C}$. Define a bilinear map $[~,~]: \mathfrak g\times \mathfrak g \longrightarrow \mathfrak g$ by $[e_1,e_3]=e_2$ and $[e_3,e_3]= e_1$, all other products of basis elements being $0$. Then $(\mathfrak g,[~,~])$ is a Leibniz algebra over $\mathbb{C}$ of dimension $3$. The Leibniz algebra $\mathfrak g$ is  nilpotent and is denoted by $\lambda_6$  in the classification of three dimensional nilpotent Leibniz algebras, see \cite{L1,A3}. 
\end{exam}

\begin{defn}
A morphism $\phi : (\mathfrak g_1, [~,~]_1) \rightarrow (\mathfrak g_2, [~,~]_2)$ of Leibniz algebras is a Linear map which preserves the brackets, that is, 
$$\phi ([x,y]_1) = [\phi (x), \phi (y)]_2,~~x, y \in \mathfrak g_1.$$
\end{defn}

Recall that the homology $HL_\ast(\mathfrak g)$ of a Leibniz algebra is defined as follows. To any Leibniz algebra $\mathfrak g$ there is an associated chain complex 
$$CL_\sharp(\mathfrak g): \cdots \rightarrow \mathfrak g^{\otimes n}\stackrel{d}{\rightarrow} \mathfrak g^{\otimes (n-1)}\stackrel{d}{\rightarrow}\cdots \stackrel{d}{\rightarrow} \mathfrak g^{\otimes 2} \stackrel{d}{\rightarrow} \mathfrak g $$
where
\begin{align}\label{boundary-map}
d(x_1, \ldots ,x_n) = \sum_{1\leq i < j \leq n} (-1)^j (x_1, \ldots ,x_{i-1}, [x_i, x_j], x_{i+1}, \ldots , \hat{x}_j, \ldots , x_n).
\end{align}
The map $d$ satisfy $d^2 = 0$ (\cite{L0}). The homology groups of this complex are denoted by $HL_n(\mathfrak g),$ $n \geq 1.$

Next, recall that the Leibniz cohomology $HL^\ast(\mathfrak g; A)$ of a Leibniz algebra $\mathfrak g$ with coefficients in an associative commutative $\mathbb{K}$-algebra $A$ is defined as follows.

Set $CL^n(\mathfrak g; A) = \text{Hom}_{\mathbb K}(\mathfrak g^{\otimes n}, A).$ Then define 
$$\delta : CL^n(\mathfrak g; A)\rightarrow CL^{n+1}(\mathfrak g; A)$$ by $\delta (c) = c\circ d,$ $c \in CL^n(\mathfrak g; A),$ where $d : \mathfrak g^{\otimes(n+1)} \rightarrow \mathfrak g^{\otimes n}$ is the boundary map (\ref{boundary-map}). Explicitly, for any $c \in CL^n(\mathfrak g; A)$ and $(x_1, \ldots ,x_{n+1}) \in \mathfrak g ^{\otimes (n+1)},$
$\delta (c)(x_1, \ldots ,x_{n+1})$ is given by the expression
\begin{align}\label{coboundary-map}
\sum_{1\leq i < j \leq n+1} (-1)^j c(x_1, \ldots ,x_{i-1}, [x_i, x_j], x_{i+1}, \ldots , \hat{x}_j, \ldots , x_{n+1}).
\end{align}
Clearly, $\delta^2 =0$ as $d^2=0,$ and therefore, $(CL^\sharp(\mathfrak g; A), \delta)$ is a cochain complex. Its homology groups are called Leibniz cohomology groups of $\mathfrak g$ with coefficients in $A$ and denoted by $HL^\ast(\mathfrak g ; A).$

\section{Group actions on Leibniz algebras}
The purpose of this section is to introduce finite group actions on Leibniz algebras and provide examples of group actions.

\begin{defn}\label{definition-group-action}
Let $\mathfrak g$ be a Leibniz algebra and $G$ be a finite group. The group $G$ is said to act from the left if there exists a function 
$$\phi : G\times \mathfrak g \rightarrow \mathfrak g,~~ (g, x) \mapsto \phi (g, x) = gx$$ satisfying the following conditions.
\begin{enumerate}
\item For each $g\in G$ the map $x\mapsto gx,$ denoted by $\psi_g$ is linear.
\item $ex= x$ for all $x \in \mathfrak g$, where $e \in G$ is the group identity.
\item $g_1(g_2x) = (g_1g_2)x$ for all $g_1, g_2 \in G$ and $x \in \mathfrak g$.
\item $g[x, y] = [gx, gy]$ for all $g\in G$ and $x, y \in \mathfrak g.$
\end{enumerate}
When $\mathfrak g= \{\mathfrak g_i\}_{i\geq 0}$ is a graded Leibniz algebra, we further assume that for each $g \in G$ the map $\psi_g$ is graded linear of degree $0$.  
\end{defn}

The following is an equivalent formulation of the above definition.
\begin{prop}
Let $G$ be a finite group and $\mathfrak g$ be a Leibniz algebra. Then $G$ acts on $\mathfrak g$ if and only if there exists a group homomorphism 
$$\psi : G \rightarrow \text{Iso}_{Leib} (\mathfrak g, \mathfrak g),~~g \mapsto \psi(g)=\psi_g$$ from the group $G$ to the group of Leibniz algebra isomorphisms from $\mathfrak g$ to $\mathfrak g$, where $\psi_g (x) = gx$ is the left translation by $g.$  
\end{prop}
\begin{remark}
Let ${\mathbb K}[G]$ be the group ring. If a Leibniz algebra $\mathfrak g$ is equipped with an action of $G$ then $\mathfrak g$ may be viewed as a ${\mathbb K}[G]$-module.
\end{remark}

Observe that if $\mathfrak g$ is a Leibniz algebra equipped with an action of a group $G$ as above, then for every subgroup $H \subset G,$ the H-fixed point set $\mathfrak g^H$ is defined by
$$\mathfrak g^H = \{x \in \mathfrak g : hx = x ~~\forall h \in H\}.$$ Clearly, for every subgroup $H \subset G,$ $\mathfrak g^H$ is sub Leibniz algebra of $\mathfrak g$. Moreover, note that if $H, K$ are subgroups of $G$ with $g^{-1}Hg \subset K,~~g\in G,$ then the Leibniz algebra homomorphism $\psi_g$ maps $\mathfrak g^K$ to $\mathfrak g^H$.

\begin{exam}\label{example-3}
Let $V$ be $\mathbb K$-module which is a representation space of a finite group $G.$ On 
$$\bar{T}(V) = V\oplus V^{\otimes 2}\oplus \cdots \oplus V^{\otimes n}\oplus \cdots $$ there is a unique bracket that makes it into a Leibniz algebra and verifies 
$$v_1 \otimes v_2 \otimes  \cdots \otimes v_n = [\cdots [[v_1,v_2],v_3],\cdots ,v_n]~\mbox{for}~v_i\in V~\mbox{and}~i=1,\cdots,n.$$ This is the free Leibniz algebra over the $\mathbb{K}$-module $V$. The linear action of $G$ on $V$ extends naturally to an action on $\bar{T}(V)$ and the bracket defined above satisfies the conditions of Definition \ref{definition-group-action}. Thus, $(G, \bar{T}(V))$ is an action of $G$ on the free Leibniz algebra $\bar{T}(V)$.  
\end{exam}

\begin{exam}\label{example-4}
Let $(\mathfrak g,d)$ be a differential Lie algebra with the Lie bracket $[~,~]$. Assume that a finite group $G$ acts linearly on $\mathfrak g$ such that 
\begin{enumerate}
\item $[gx, gy] = g[x, y]$ for all $g \in G$ and $x, y \in \mathfrak g$,
\item $d(gx) = gd(x)$ for all $g\in G$ and $x \in \mathfrak g.$ 
\end{enumerate}
Then, the group $G$ acts on the Leibniz algebra $(\mathfrak g, [~, ~]_d),$ where $[x,y]_d:= [x,dy]$ is the derived bracket (cf. Example \ref{example-1}).
\end{exam}

Our next example is based on a specific case of the above Example.
\begin{exam}\label{example-5}
Let $V$ be a vector space over a field $\mathbb K.$  Assume that a finite group $G$ acts linearly on $V.$ Let
$$C^k(V) = \{\alpha : V \times \cdots \times V \rightarrow V| \alpha ~~~\mbox{is linear in each argument}\}$$
be the set of all $\mathbb K$-multilinear maps on $V.$ For $\alpha \in C^k(V)$ and $\beta \in C^l(V)$ let $\alpha\circ \beta \in C^{k+l-1}(V)$ be the element defined by
\begin{align*}
& \alpha\circ \beta (x_1, \ldots , x_{k+l-1})\\
& = \sum_{i=1}^k(-1)^{(i-1)(l-1)} \alpha(x_1, \ldots, x_{i-1}, \beta(x_i, \ldots , x_{i+l-1}), x_{i+1}, \ldots , x_{k+l-1})
\end{align*}
and let $[\alpha,\beta] \in C^{k+l-1}(V)$ be the Gerstenhaber bracket defined by:
$$[\alpha, \beta] := \alpha\circ \beta - (-1)^{(k-1)(l-1)} \beta\circ \alpha.$$ Let $C^\bullet(V) = \oplus_kC^k(V).$ Then, if we declare an element of $C^k(V)$ to have degree $k-1,$ the Gerstenhaber bracket defined a structure of graded Lie algebra on $C^\bullet(V) = \oplus_kC^k(V).$ 

Next observe that the linear action of $G$ on $V$ induces an action on $C^k(V)$ for each $k$ and is given by $(g, \alpha) \mapsto g(\alpha),$ where
$$g(\alpha)(x_1, \ldots , x_k) := g(\alpha (g^{-1}x_1, \ldots , g^{-1}x_k)).$$ Thus, an element $\alpha \in C^k(V)$ is invariant with respect to the above action if and only if $\alpha : V\times \cdots \times V \rightarrow V$ is an equivariant multilinear map where $G$ acts component wise on any product of $V.$ Moreover, it is straight forward to verify that the above action on the graded vector space $C^\bullet(V)$ satisfies $[g\alpha,g\beta] = g[\alpha, \beta].$   

Assume that there is an element $\alpha\in C^2(V)$ which is invariant with respect to the above action and defines an associative algebra structure on $V$. The last assertion is equivalent to assume that $[\alpha, \alpha] = 0.$ For such an element $\alpha \in C^2(V)$  we define a map $d_{\alpha} : C^\ast(V) \rightarrow C^{\ast +1}(V)$ by $d_{\alpha}(\beta) = [\alpha, \beta].$ Then, $d_{\alpha}:C^\ast(V)\rightarrow C^{\ast +1}(V)$ is equivariant and by graded Jacobi identity $d_{\alpha}^2 = 0.$ Thus,  $(C^\bullet (V), [~,~], d_{\alpha})$ is a differential graded Lie algebra equipped with a linear action of $G$ such that $[g\alpha,g\beta] = g[\alpha, \beta]$ for all $g\in G,$ $\alpha \in C^k(V)$ and $\beta \in C^l(V).$ On $C^\bullet (V)$ we consider the derived bracket 
$$[\beta, \gamma]_{d_{\alpha}} := [\beta, d_{\alpha}\gamma].$$ Then, $(C^\bullet (V), [~,~]_{d_{\alpha}})$ is a graded Leibniz algebra equipped with an action of the group $G$.
\end{exam}

Finally, we discuss two geometric examples.  

\begin{exam}\label{example-6}
Recall that every Lie algebra, in particular, is a Leibniz algebra as the in the presence of skew symmetry, the Leibniz identity reduces to the Jacobi identity. Let $M$ be a smooth manifold equipped with a smooth action of $G.$  For each $g\in G,$ let $l_g : M \rightarrow M$ denote the left translation by $g,$ that is, $l_g (x) = gx,~~ x \in M.$ Consider the Lie algebra $(\chi (M), [~,~])$ of vector fields on $M,$ where for vector fields $X, Y \in \chi (M)$ their Lie bracket $[X, Y]$ is a vector field which acts on smooth functions $f \in C^\infty(M)$  by $[X, Y] (f) := X(Yf) -Y(Xf).$ Define  $G \times \chi (M) \rightarrow \chi (M)$ by $(g, X) \mapsto (l_g)_\ast (X),$ where $(l_g)_\ast (X)$ is the push forward by $l_g.$ Explicitly, for $f \in C^\infty(M),$ $(l_g)_\ast (X)(f) = X(f\circ l_g).$ Then, it is easy to check that $(l_g)_\ast ([X, Y]) = [(l_g)_\ast(X),(l_g)_\ast(Y)]$ and hence $(G, \chi (M))$ is an action of $G$ on the Lie algebra $\chi (M).$     
\end{exam}

The following discussion is a prelude to our next example.

\begin{defn}\label{action-on-forms}
Let $G$ be a finite group and $M$ is a smooth $G$-manifold. Then for any $k \geq 0,$ there is an action of $G$ on the space of $k$-forms $\Omega^k (M)$ on $M$, given by
$$ G \times \Omega^k (M) \rightarrow \Omega^k (M), ~(g , \alpha)(x) := (d l_{g^{-1}})^* ~\alpha (g^{-1}x),$$
for $\alpha \in \Omega^k (M)$ and $x \in M$. Moreover, there is an action of $G$ on the space of vector fields $\mathcal{X} (M)$ on $M$, namely,
$$ G \times \mathcal{X} (M) \rightarrow \mathcal{X} (M), ~(g, X) (x) := (dl_g)~ X (g^{-1}x),$$
for $X \in \mathcal{X} (M)$ and $x \in M$. This action is the same as the action of $G$ on $\chi (M)$ by push forward by $l_g$ as discussed in the above example.
\end{defn}

\begin{lemma}\label{lemma-invariant}
The contraction operator and the de Rham differential operator satisfies the following properties. For any  $\alpha \in \Omega^k (M)$, $X \in \mathcal{X}(M)$ and $g \in G$,
\begin{itemize}
 \item[(i)] $i_{(g, X)} (g, \alpha) = (g, i_X \alpha)$
 \item[(ii)] $d (g, \alpha) = (g, d \alpha)$
 \item[(iii)] $\mathcal{L}_{(g, X)} (g, \alpha) = ( g, \mathcal{L}_X \alpha)$
 \item[(iv)] If $\Pi $ is a $G$-invariant $(k+1)$-vector field, that is, $\Pi (gx) = (dl_g) ~ \Pi (x)$, for all $g \in G$ and $x \in M$, then $\Pi^\sharp (g, \alpha) = (g, \Pi^\sharp \alpha).$
\end{itemize}
\end{lemma}

\begin{proof}
(i) For any $x \in M$ and $X_1, \ldots, X_{k-1} \in T_xM$, we have
\begin{align*} 
&(i_{(g, X)} (g, \alpha))(x) (X_1, \ldots, X_{k-1}) \\
=~& (g, \alpha) (x) \big( (g, X)(x) , X_1,   \ldots, X_{k-1}   \big) \\
=~& (dl_{g^{-1}})^* \alpha (g^{-1} x) ~\big(  (dl_g) X(g^{-1}x), X_1, \ldots, X_{k-1}     \big) \\
=~&  \alpha (g^{-1} x) ~  \big( X(g^{-1}x), (dl_{g^{-1}}) X_1, \ldots, (dl_{g^{-1}}) X_{n-1}  \big) \\
=~& (i_X \alpha )(g^{-1}x) \big(  (dl_{g^{-1}}) X_1, \ldots, (dl_{g^{-1}}) X_{n-1}   \big) \\
=~& ((dl_{g^{-1}})^*~  (i_X \alpha) (g^{-1}x)) (X_1, \ldots, X_{k-1})\\
=~& (g, i_X \alpha)(x) (X_1, \ldots, X_{k-1}).
\end{align*}

(ii) Note that the action of $G$ on $C^\infty(M)$ is given by $(g, f) := (l_{g^{-1}})^* f$. Therefore, for any $x \in M$, 
$$ (g, df)(x) = (dl_{g^{-1}})^* ~df (g^{-1}x) = d ( (l_{g^{-1}})^* f ) (x) =  d (g, f) (x).$$
The result now follows from the observation that for any $\alpha, \beta \in \Omega^\bullet (M)$,
\begin{align*}
 &(g, \alpha \wedge \beta)(x) \\
=~& (dl_{g^{-1}})^* ~(\alpha \wedge \beta)(g^{-1}x) \\
=~& (dl_{g^{-1}})^*  ~ (   \alpha (g^{-1}x) \wedge \beta (g^{-1}x)) \\
=~& (dl_{g^{-1}})^*  \alpha (g^{-1}x) \wedge (dl_{g^{-1}})^* \beta (g^{-1}x) = \big((g, \alpha) \wedge (g, \beta) \big)(x).
\end{align*}

(iii) It follows from part (i) and (ii) and the Cartan magic formula
$$ \mathcal{L}_X = i_X d + d i_X .$$

(iv) For any $x \in M$ and $\beta_x \in T_x^*M$, we have
\begin{align*}
 &\langle \Pi^\sharp (g, \alpha) (x) , \beta_x   \rangle \\
=~& (i_{(g, \alpha) (x)} \Pi(x)) (\beta_x) \\
=~& \Pi (x) \big( (dl_{g^{-1}})^* \alpha (g^{-1} x) , \beta_x \big) \\
=~& ( dl_g ) ~\Pi (g^{-1} x) \big( (dl_{g^{-1}})^* \alpha (g^{-1} x) , \beta_x \big)   \hspace*{1cm} \text{(since $\Pi$ is $G$-invariant)} \\
=~& \Pi (g^{-1} x) \big( \alpha (g^{-1} x) ,  ( dl_g )^* \beta_x \big) \\
=~& \langle  (dl_g) ~(i_{\alpha (g^{-1} x)} \Pi (g^{-1} x)), \beta_x  \rangle \\
=~& \langle  (dl_g) ~( \Pi^\sharp \alpha) (g^{-1} x), \beta_x  \rangle = \langle (g, \Pi^\sharp \alpha) (x), \beta_x \rangle.
\end{align*}
\end{proof}

Recall that a {\it Nambu-Poisson} manifold is a generalization of the notion of {\it Poisson} manifolds and is defined as follows \cite{book-vaisman}, \cite{ilmp}. 
\begin{defn}\label{nambu-poisson manifold}
Let $M$ be a smooth manifold.  A  {\it{Nambu-Poisson bracket}} of order $n$ ($2\leq n \leq ~\mbox{dim}~ M$) on $M$ is an $n$-multilinear
mapping 
$$\{,\ldots,\}\colon C^\infty(M)\times \cdots\times C^\infty(M)
\longrightarrow C^\infty(M)$$ satisfying the following conditions: 
\begin{enumerate}
\item Skew-symmetric: $ \{f_1,\ldots,f_n\}= \textup{sign}(\sigma)\{f_{\sigma(1)},\ldots,f_{\sigma(n)}\}$
for any $\sigma\in \Sigma_n;$
\item Leibniz rule: $\{fg,f_2,\ldots,f_n\}= f\{g ,f_2,\ldots,f_n\}+ g\{f, f_2, \ldots,f_n\};$
\item Fundamental identity: $$\{f_1,\ldots,f_{n-1},\{g_1,\ldots,g_n\}\}=
\displaystyle{\sum_{i=1}^n}\{g_1,\ldots, g_{i-1}, \{f_1,\ldots,f_{n-1},g_i\},\ldots,g_n\}$$
\end{enumerate}
for $f_i,g_j,f, g\in C^\infty(M).$ The pair $(M, \{~, \ldots, ~\})$ is called a {\it{Nambu-Poisson manifold of order $n$}}.  
\end{defn}
\begin{remark}
Recall \cite{book-vaisman} that Poisson manifolds are Nambu-Poisson manifolds of order $2$.  
\end{remark}

\noindent Given a Nambu-Poisson bracket on $M$, there exists an $n$-vector field $P\in\Gamma (\Lambda^n TM),$ called the Nambu-Poisson tensor corresponding to the given bracket, and is defined by $P(df_1,\ldots,df_n)=\{f_1,\ldots,f_n\},$ for $f_1,\ldots,f_n\in C^\infty(M)$. Note that $P$ induces a bundle map $P^\sharp: \Lambda^{n-1}T^\ast M\rightarrow TM$ given by
$$\big\langle \beta, P^\sharp (\alpha_1 \wedge \cdots \wedge \alpha_{n-1})\big\rangle = P (\alpha_1, \ldots , \alpha_{n-1}, \beta),$$
for all $\alpha_1, \ldots , \alpha_{n-1}, \beta \in \Omega^1(M).$

Recall the following definition from \cite{wade}.
\begin{defn}
 A {\it (left) Leibniz algebroid} over a smooth manifold $M$ is a smooth vector bundle $A$ over $M$ together with a bracket $[~,~]$ on the space $\Gamma{A}$ of smooth sections of $A$ and a bundle map $\rho : A \rightarrow TM$, called the {\it anchor} 
such that the bracket satisfies
\begin{enumerate}
\item (left) Leibniz identity: $ [X, [Y, Z]] = [[X, Y], Z] + [Y, [X, Z]],~~X, Y, Z \in \Gamma{A};$
\item $[X,fY] = f[X, Y] + (\rho(X)f)Y$;
\item $\rho ([X, Y]) = [\rho (X), \rho (Y)], ~~X, Y \in \Gamma A,~~f\in C^\infty(M).$
\end{enumerate}
\end{defn}

We recall the following result from \cite{wade}, \cite{ilmp}.

Let $M$ be a Nambu-Poisson manifold of order $n$ with the corresponding Nambu tensor $\Pi$. Then the bundle $\bigwedge^{n-1} T^*M$ carries a Leibniz algebroid structure with bracket
$$ [\alpha, \beta] := \mathcal{L}_{\Pi^\sharp \alpha} \beta - i_{\Pi^\sharp \beta} d \alpha$$
on the space $\Omega^{n-1}(M)$ of $(n-1)$-forms on $M$ and the anchor is given by the bundle map $\Pi^\sharp.$ Thus $\Omega^{n-1}(M)$ is a Leibniz algebra with respect this bracket.

\begin{exam}\label{example-7}
Suppose $G$ is a finite group. Let $M$ be a Nambu-Poisson manifold of order $n$ with the associated Nambu tensor $\Pi.$ Assume that $G$ acts smoothly on $M$ and $\Pi$ is $G$-invariant with respect to the action of $G$ as defined in Definition \ref{action-on-forms}. Then, it follows from Lemma \ref{lemma-invariant} that for any $\alpha, \beta \in \Omega^{n-1} (M)$
and $g \in G$,
\begin{align*}
 [(g, \alpha) , (g, \beta)] =~& \mathcal{L}_{\Pi^\sharp (g, \alpha)} (g, \beta) - i_{\Pi^\sharp (g, \beta)} d (g, \alpha)\\
=~& \mathcal{L}_{(g, \Pi^\sharp \alpha)} (g, \beta) - i_{(g, \Pi^\sharp \beta)} d (g, \alpha) = (g, [\alpha, \beta]).
\end{align*}
Thus, with the bracket $[~,~]$ as defined above $\Omega^{n-1}(M)$ is a Leibniz algebra equipped with the action of the given group $G$. 

\end{exam}

\section{Equivariant cohomology of a Leibniz algebra equipped with a group action}
In this section we introduce equivariant cohomology groups of a Leibniz algebra equipped with an action of a finite group following \cite{bredon67},\cite{illman}.

Let $G$ be a finite group. Recall that the category of canonical orbits of $G$, denoted by $O_G,$ is a category whose  objects are left cosets $G/H$, as $H$ runs over the all subgroups of $G$. Note that the group $G$ acts on the set $G/H$  by left translation. A morphism from $G/H$ to $G/ K$ is a $G$-map. Recall that such a morphism determines and is determined by a subconjugacy relation $g^{-1}Hg\subseteq K$ and is given by $\hat{g}(eH)=gK$. We denote this morphism by $\hat{g}$ \cite{bredon67}.

\begin{defn}
An $O_G$-module is a contravariant functor  $M : O_G \rightarrow \text{\bf Mod},$ where {\bf Mod} is the category of modules over ${\mathbb K}$. The category whose objects are $O_G$-modules and with morphisms the natural transformations between $O_G$-modules is an abelian category denoted by ${\mathcal C}_G.$  Let {\bf Comm} be the category of  associative commutative algebras over $\mathbb{K}.$  An $O_G$-algebra is a contravariant functor $A : O_G \rightarrow \text{ \bf Comm}.$ Similarly, an $O_G$-Leibniz algebra is a contravariant functor $L: O_G \rightarrow {\bf Leib}.$ If $\mathfrak g$ is a Leibniz algebra equipped with an action of $G,$ then we have contravariant functor $\Phi \mathfrak g : O_G \rightarrow \text{ \bf Leib},$ given by $\Phi \mathfrak g (G/H) = \mathfrak g^H$ and for a morphism $\hat{g}: G/H \rightarrow G/ K$ corresponding to a sub conjugacy relation $g^{-1}Hg\subseteq K,$ $\Phi \mathfrak g (\hat{g}) = \psi_g : \mathfrak g^K \rightarrow \mathfrak g^H.$ Thus, $\Phi \mathfrak g$ is an $O_G$-Leibniz algebra, which will be referred to as the $O_G$-Leibniz algebra associated to $\mathfrak g$.
\end{defn}

We now proceed to define the notion of equivariant cohomology of a Leibniz algebra $\mathfrak g$ equipped with an action of a finite group $G.$ Let $A : O_G \rightarrow \text{ \bf Comm}$ be an $O_G$-algebra. Let the product in $A(G/H)$ is denoted by $\mu_H : A(G/H) \otimes A(G/H) \rightarrow A(G/H).$ Note that for every morphism $\hat{g}: G/H \rightarrow G/ K$ in $O_G$ corresponding to a sub conjugacy relation $g^{-1}Hg\subseteq K,$ we have 
$$\mu_H \circ (A(\hat{g})\otimes A(\hat{g})) = A(\hat{g}) \circ \mu_K.$$

\begin{defn}
For every $n \geq 1,$ let $\underline{CL}_n(\mathfrak g)$ be the $O_G$-module defined by $\underline{CL}_n(\mathfrak g)(G/H) := CL_n(\mathfrak g^H)= (\mathfrak g^H)^{\otimes n}$ and 
$\underline{CL}_n(\mathfrak g)(\hat{g}): CL_n(\mathfrak g^K)\rightarrow CL_n(\mathfrak g^H)$ is given by $(\psi_g)^{\otimes n}.$ The boundary map (\ref{boundary-map}) induces a natural transformation $\underline{d} : \underline{CL}_{n+1}(\mathfrak g)\rightarrow \underline{CL}_n(\mathfrak g),$ where $\underline{d}(G/H) = d_H$ is the boundary map for the Leibniz algebra $\mathfrak g ^H.$ Clearly, $\underline{d}\circ \underline{d} = 0.$ This gives a chain complex in the abelian category of $O_G$-modules. Let 
$$CL^n_G(\mathfrak g; A):= \text{Hom}_{\mathcal C _G}( \underline{CL}_n(\mathfrak g), A).$$ We have an induced homomorphism
$\delta :CL^n_G(\mathfrak g; A)\rightarrow CL^{n+1}_G(\mathfrak g; A)$ given by $\delta (c) := c\circ \underline{d}$ for any natural transformation $c \in CL^n_G(\mathfrak g; A).$ Thus, we have a cochain complex $CL^\sharp_G(\mathfrak g; A)=\{CL^n_G(\mathfrak g; A), \delta\}.$ The homology groups of this cochain complex are equivariant Leibniz algebra cohomology groups of $(G, \mathfrak g)$ with coefficients in $A$ and denoted by $HL_G^n(\mathfrak g; A).$ 
\end{defn}

The following is an equivalent formulation of the groups $HL_G^n(\mathfrak g; A).$

Set $S^n (\mathfrak g; A) = \oplus_{H < G}CL^n(\mathfrak g^H; A(G/H))$ and define 
$$\delta : S^n (\mathfrak g; A)\rightarrow S^{n+1} (\mathfrak g; A)$$ by $\delta = \oplus_{H < G}\delta _H,$ where $\delta_H: CL^n(\mathfrak g^H; A(G/H))\rightarrow CL^{n+1}(\mathfrak g^H; A(G/H))$ is the non-equivariant coboundary map (\ref{coboundary-map}) for the Leibniz algebra $\mathfrak g^H.$
Clearly, $\{S^n (\mathfrak g; A), \delta\}$ is a cochain complex. We define a subcomplex of this cochain complex as follows. 
\begin{defn}
A cochain $c=\{c_H\} \in S^n (\mathfrak g; A)$ is said to be invariant under the action of $G$ if for every morphism $\hat{g} : G/H \rightarrow G/K$, corresponding to a subconjugacy relation $g^{-1}Hg\subset K$ following holds:
$$c_H \circ(\psi_g)^{\otimes n}= A(\hat{g})\circ c_K.$$
\end{defn}

\begin{lemma}
The set of all invariant $n$-cochains is a a subgroup $S^n _G(\mathfrak g; A)$ of $S^n (\mathfrak g; A).$ If $c=\{c_H\} \in S^n (\mathfrak g; A)$ is invariant then $\delta (c)=\{\delta_H (c_H)\}\in S^{n+1} (\mathfrak g; A)$ is an invariant $(n+1)$-cochain.
\end{lemma}

\begin{proof}
It is clear that $S^n _G(\mathfrak g; A)$ is a subgroup of $S^n(\mathfrak g; A)$ as  $A(\hat{g})$ is a homomorphism for every $\hat{g} :G/H \rightarrow G/K.$
Let $c=\{c_H\} \in S^n (\mathfrak g; A)$ be invariant and $\hat{g} : G/H \rightarrow G/K$ be a morphism in $O_G$ corresponding to a subconjugacy relation $g^{-1}Hg\subset K$ . Thus, for every $(x_1, \ldots, x_n) \in (\mathfrak g ^K)^{\otimes n}$ we have
\begin{align}\label{equality-one}
& c_H (\psi_g(x_1), \ldots, \psi_g(x_n))= A(\hat{g})(c_K (x_1, \ldots , x_n)).
\end{align}
Next, recall from (\ref{coboundary-map}) that
\begin{align*} 
& \delta_H (c_H)(\psi_g(x_1), \ldots , \psi_g(x_{n+1}))\\
&= \sum_{1\leq i < j \leq n+1} (-1)^j c_H(gx_1, \ldots, gx_{i-1}, [gx_i, gx_j], \ldots , \widehat{gx_j}, \ldots , gx_{n+1})\\
&= \sum_{1\leq i < j \leq n+1} (-1)^j A(\hat{g})c_K(x_1, \ldots ,x_{i-1}, [x_i, x_j], x_{i+1}, \ldots , \hat{x}_j, \ldots , x_{n+1})\\
&= A(\hat{g})\delta _K(c_K)( x_1, \ldots, x_{n+1}).
\end{align*}
Therefore, $\delta_H (c_H)\circ (\psi_g)^{\otimes n+1} = A(\hat{g})\circ \delta _K(c_K)$ and hence,
$$\{\delta_H(c_H)\} \in S^{n+1}_G (\mathfrak g; A).$$
\end{proof}

Thus, we have a cochain subcomplex $S^\sharp_G(\mathfrak g ; A) = \{S_G^n (\mathfrak g; A), \delta\}.$

\begin{thm}\label{equivalent-formulation}
Let $\mathfrak g$ be a Leibniz algebra with a given action of $G.$ For any $O_G$-algebra $A$ we have an isomorphism
$$H_n(S^\sharp_G(\mathfrak g ; A)) \cong HL^n_G(\mathfrak g; A)$$ for all $n$. 
\end{thm}

\begin{proof}
Let $F \in CL^n_G(\mathfrak g; A).$ Note that the collection $\{F(G/H): ({\mathfrak g}^H)^{\otimes n} \rightarrow A(G/H)\}$ of non-equivariant $n$-cochain of $\mathfrak g ^H,$ as $H$ varies over subgroups of $G,$ is invariant. Because, for every morphism $\hat{g} : G/H \rightarrow G/K$ of $O_G,$ we have $A(\hat{g}) \circ F(G/K) = F(G/H) \circ (\psi_g)^{\otimes n}$ by naturality of $F.$ Thus, we have a map 
$$\alpha_n : CL^n_G(\mathfrak g; A) \rightarrow S^n_G(\mathfrak g; A)$$ defined by $F \mapsto \alpha_n (F): = \{F(G/H)\}$ for all $n.$ We claim that $\alpha = \{\alpha_n\}$ is a cochain map. Let $F \in CL^n_G(\mathfrak g; A).$ Then, $(\delta \circ \alpha)(F) = \{\delta_H F(G/H)\}.$ On the other hand, $(\alpha \circ \delta)(F) = \alpha (\delta F) = \alpha (F\circ \underline{d}) = \{(F\circ \underline{d})(G/H)\} = \{F(G/H) \circ d_H\}.$ Since $\delta_H F(G/H) = F(G/H) \circ d_H$ the result follows. Thus, $\alpha$ induces a homomorphism
$$\bar{\alpha}_n : HL^n_G(\mathfrak g; A) \rightarrow H_n(S^\sharp_G(\mathfrak g ; A))$$ for all $n$.

Finally, observe that the cochain map $\alpha = \{\alpha_n\}$ is a cochain isomorphism with inverse $\beta =\{\beta_n\}$ defined as follows:
$$\beta : S^n _G(\mathfrak g; A) \rightarrow CL^n_G(\mathfrak g; A), ~~\{c_H\}\mapsto C,$$ where the natural transformation $C: \underline{CL}_n(\mathfrak g) \rightarrow A$ is given by $C(G/H) := c_H.$ The naturality of $C$ follows from the invariance of $\{c_H\}.$ Thus, $\bar{\alpha}_n$ is an isomorphism for all $n$.   
\end{proof}

\section{Equivariant Leibniz cohomology as zinbiel algebra}

In \cite{L3}, J.-L. Loday introduced zinbiel algebras and proved that for any Leibniz algebra $\mathfrak g$ the graded Leibniz cohomology with coefficients in a commutative, associative algebra admits a graded product which makes it a graded zinbiel algebra. The aim of this section is to prove an equivariant version of this result. Explicitly, for a Leibniz algebra $\mathfrak g$ equipped with an action of $G$, we prove that equivariant graded Leibniz cohomology $HL^\ast_G(\mathfrak g, A)$ as introduced in the previous section also admits a graded zinbiel algebra structure. To show this, we use the equivalent formulation (Theorem \ref{equivalent-formulation}) of equivariant cohomology to define a cup-product operation which is based on the product defined at the cochain level of the fixed points Leibniz algebras $\{\mathfrak g^H\}_{H< G}.$      

Recall the following definition from \cite{L3, L4}.
\begin{defn} 
A dual Leibniz algebra or a Zinbiel algebra is a $\mathbb{K}$-vector space $R$ equipped with a bilinear map
$$(--): R \times R \rightarrow R$$ satisfying the relation
\begin{align}\label{zinbiel-relation}
((rs)t) = (r(st)) + (r(ts)),  \forall~r, s, t \in R.
\end{align} 
A graded Zinbiel algebra is a graded $\mathbb{K}$-vector space $R$ equipped with a graded bilinear map 
$$(--): R \times R \rightarrow R$$ satisfying the relation
\begin{align}\label{graded-zinbiel-relation}
((rs)t) = (r(st)) + (-1)^{|t||s|}(r(ts)), 
\end{align}
for all homogeneous elements $r, s, t \in R.$ 
\end{defn}

Let $S_n$ be the permutation group of $n$ elements $1, \ldots, n.$ A permutation $\sigma \in S_n$ is called a $(p, q)$-shuffle if $p+q = n$ and 
$$\sigma (1) < \cdots <\sigma (p) ~~\mbox{and}~~~\sigma (p+1) < \cdots <\sigma (p+q).$$ 
In the group algebra $\mathbb K [S_n],$ let $sh_{p,q}$ be the element $$sh_{p,q}: = \sum_{\sigma} \sigma,$$ where the summation is over all $(p, q)\mbox{-shuffles}.$

For any vector space $V$ we let $\sigma \in S_n$ act on $V^{\otimes n}$ by
$$\sigma(v_1\ldots v_n) = (v_{\sigma^{-1}(1)}\ldots v_{\sigma^{-1}(n)}),$$ where the generator $v_1\otimes \cdots \otimes v_n$ of $V^{\otimes n}$ is denoted by $v_1\ldots v_n.$
The free Zinbiel algebra over the vector space $V$ is $\bar{T}(V) = \oplus_{n\geq 1}V^{\otimes n}$ equipped with the following product
$$(v_0\ldots v_p)(v_{p+1}\ldots v_{p+q}) = v_0sh_{p,q}(v_1\ldots v_{p+q})= (Id_1\otimes sh_{p,q})(v_0\ldots v_{p+q}).$$
Here $Id_1$ is the identity on the first factor.

Note that the linear map from $\mathbb K[S_n]$ to itself induced by $\sigma \mapsto \mbox{sgn}(\sigma)\sigma^{-1}$ for $\sigma \in S_n$ is an anti-homomorphism. Let us denote the image of $\alpha \in \mathbb K[S_n]$  under this map by $\tilde{\alpha}.$ 

Let $\mathfrak g$ be a Leibniz algebra equipped with a given action of a finite group $G.$ Let $\Phi \mathfrak g$ be the corresponding $O_G$-Leibniz algebra. Denote by $\Phi \mathfrak g^{\otimes p+q}$ the $O_G$-vector space given by  
$$\Phi \mathfrak g^{\otimes p+q}(G/H) =(\Phi \mathfrak g(G/H))^{\otimes p+q}= (\mathfrak g^H)^{\otimes p+q}$$ for objects $G/H$ in $O_G$ and for a morphism $\hat{g} : G/H \rightarrow G/K,$ $$\Phi \mathfrak g^{\otimes p+q}(\hat{g})= (\Phi \mathfrak g (\hat{g}))^{\otimes p+q}.$$

For any non-negative integers $p$ and $q$ we define a natural transformation
\begin{align}\label{equivariant-linear-map}
\rho_{p,q}: = Id_1\otimes \widetilde{sh}_{p-1,q} = \Phi \mathfrak g^{\otimes p+q} \rightarrow \Phi \mathfrak g^{\otimes p+q}.
\end{align}

Explicitly, for every object $G/H \in O_G,$ the linear map 
$$\rho_{p,q}(G/H): (\mathfrak g ^H)^{\otimes p+q} \rightarrow (\mathfrak g ^H)^{\otimes p+q}$$ is given by
\begin{align}\label{explicit-above-map}
\rho_{p,q}(G/H)(x_1\ldots x_{p+q})= \sum_\sigma \mbox{sgn}(\sigma)(x_1x_{\sigma (2)}\ldots x_{\sigma (p+q)}),
\end{align}
where the above sum is over all $(p-1, q)$-shuffles $\sigma$. 

Let $\tau_{p,q}:\Phi \mathfrak g^{\otimes p+q} \rightarrow \Phi \mathfrak g^{\otimes p+q}$ be the natural transformation defined as follows.
For any object $G/H$ in $O_G$ and for generators $x=v_1\ldots v_p \in (\mathfrak g ^H)^{\otimes p}$ and $y=v_{p+1}\ldots v_{p+q} \in (\mathfrak g ^H)^{\otimes q},$ 
$\tau_{p,q}(G/H)(xy) = yx$ with the obvious definition on morphisms in $O_G.$ Then, for non-negative integers $p,~~q,~~r,$ we have the following equality 
\begin{align}\label{relation-required-for-proof}
(\rho_{p,q}\otimes Id_r) \circ \rho_{p+q, r} = (Id_p\otimes \rho_{q,r} + (-1)^{rq}\circ \tau_{r,q} \circ \rho_{r,q})\circ \rho_{p,q+r}
\end{align}
(cf. \cite{L3}).

Let $A : O_G \rightarrow \text{ \bf Comm}$ be an $O_G$-algebra. Let the product in $A(G/H)$ is denoted by $\mu_H : A(G/H) \otimes A(G/H) \rightarrow A(G/H).$ Note that for every morphism $\hat{g}: G/H \rightarrow G/ K$ in $O_G$ corresponding to a sub conjugacy relation $g^{-1}Hg\subseteq K,$ $A(\hat{g}): A(G/K) \rightarrow A(G/H)$ is an algebra map and hence, we have 
$$\mu_H \circ (A(\hat{g})\otimes A(\hat{g})) = A(\hat{g}) \circ \mu_K.$$ In other words, 
$$\mu : A\times A \rightarrow A,~~\mu(G/H) = \mu_H$$ is a natural transformation. 
\begin{defn}\label{equivariant-cup-product}
For $c =\{c_H\} \in S^p_G(\mathfrak g; A),~p>0$ and $d =\{d_H\} \in S^q_G(\mathfrak g; A),~q>0,$ we define   $$c\cup d :=\{c_H \cup d_H\}=\{\mu_H\circ (c_H \otimes d_H) \circ \rho_{p,q}(G/H)\}.$$
\end{defn} 

Clearly, $c\cup d \in S^n (\mathfrak g; A).$ We claim that $c\cup d \in S^{p+q}_G(\mathfrak g; A).$ Thus, we need to prove that $c\cup d$ is invariant. Let $\hat{g} : G/H \rightarrow G/K,~~g^{-1}Hg \subset K$ be a morphism in $O_G$. We need to check that
$$(c_H\cup d_H)\circ (\psi_g)^{\otimes p+q} = A(\hat{g})\circ (c_K\cup d_K).$$
Note that
\begin{align*}
& A(\hat{g})(c_K\cup d_K)\\
& = A(\hat{g})\circ \mu_K\circ(c_K\otimes d_K)\circ \rho_{p,q}(G/K)\\
& = \mu_H \circ (A(\hat{g})\otimes A(\hat{g}))\circ (c_K\otimes d_K)\circ \rho_{p,q}(G/K)\\
& = \mu_H \circ (A(\hat{g})\circ c_K \otimes A(\hat{g})\circ d_K)\circ \rho_{p,q}(G/K)\\
& = \mu_H \circ ( c_H\circ (\psi_g)^{\otimes p}\otimes d_H\circ (\psi_g)^{\otimes q})\circ \rho_{p,q}(G/K)~~\mbox{(by invariance of $c$ and $d$)}\\
& = \mu_H \circ (c_H \otimes d_H) \circ (\psi_g)^{\otimes p+q}\circ \rho_{p,q}(G/K)\\
& = \mu_H \circ (c_H \otimes d_H) \circ \rho_{p,q}(G/H) \circ (\psi_g)^{\otimes p+q}~~\mbox{(by naturality of $\rho$)}\\
& = (c_H\cup d_H)\circ (\psi_g)^{\otimes p+q}.
\end{align*} 

Next, note that the cup-product operation is well-defined. This is because, 
\begin{align*}
\delta (c\cup d) & = \{\delta_H (c_H\cup d_H)\}\\
& = \{ \delta_H(c_H)\cup d_H + (-1)^{|c_H|}c_H\cup \delta_H(d_H)\}~~\mbox{(by the non-equivariant case)}\\
& = \{ \delta_H(c_H)\cup d_H\} + (-1)^{|c_H|}\{c_H\cup \delta_H(d_H)\}\\
& = \delta (c)\cup d + (-1)^{|c|}c \cup \delta (d)~~\mbox{( since $(-1)^{|c_H|} = (-1)^{|c|}$)}.
\end{align*}

Let $[a] \in HL^p_G(\mathfrak g; A),$ $[b] \in HL^q_G(\mathfrak g; A)$ and $[c] \in HL^r_G(\mathfrak g; A).$ We claim
\begin{align}\label{required-relation}
([a]\cup [b]) \cup [c] = [a]\cup ([b] \cup [c]) + (-1)^{qr}[a]\cup ([c] \cup [b]).
\end{align}
To prove the relation (\ref{required-relation}) we proceed as follows. We choose representative cocycles $a=\{a_H\}$, $b=\{b_H\}$ and $c= \{c_H\}.$ By Definition \ref{equivariant-cup-product}, $c\cup b :=\{c_H \cup b_H\}=\{\mu_H\circ (c_H \otimes b_H) \circ \rho_{r,q}(G/H)\}.$ Since $A(G/H)$ is commutative for every object $G/H \in O_G,$ we have
$$\mu_H\circ (c_H \otimes b_H) \circ \rho_{r,q}(G/H) = \mu_H\circ (b_H \otimes c_H) \circ\tau_{r,q}(G/H) \circ \rho_{r,q}(G/H).$$
We pre-compose $\mu_H\circ ((a_H \otimes b_H)\otimes c_H)$ on both sides of the relation (\ref{relation-required-for-proof}) for every object $G/H$ in $O_G$ to deduce 
\begin{align*}
& \mu_H\circ ((a_H \otimes b_H)\otimes c_H) \circ (\rho_{p,q}(G/H)\otimes Id_r)\circ \rho_{p+q, r}(G/H)\\
& = \mu_H\circ ((a_H \otimes b_H)\otimes c_H) \circ (Id_p\otimes \rho_{q,r}(G/H))\circ \rho_{p,q+r}(G/H)\\  
& + (-1)^{rq}\mu_H\circ ((a_H \otimes b_H)\otimes c_H)\circ (Id_p \otimes \tau_{r,q}(G/H) \circ \rho_{r,q}(G/H))\circ \rho_{p,q+r}(G/H).
\end{align*}
Thus, for every object $G/H$ in $O_G,$ we have
$$(a_H\cup b_H)\cup c_H = a_H (b_H\cup c_H) + (-1)^{rq}a_H(c_H\cup b_H).$$
Therefore, 
$$(a\cup b)\cup c = a (b\cup c) + (-1)^{|c||b|}a(c\cup b).$$

\begin{thm}
Given a Leibniz algebra $\mathfrak g$ equipped with an action of a finite group $G$ and an $O_G$-algebra $A,$ the graded equivariant cohomology $HL^\ast_G(\mathfrak g; A)$ is a graded zinbiel algebra with respect to the cup-product operation (\ref{equivariant-cup-product}).
\end{thm}

{\bf Acknowledgements:} The second author would like to thank Professor Pratulananda Das of Jadavpur University for his guidance and support.   

\mbox{ }\\

\providecommand{\bysame}{\leavevmode\hbox to3em{\hrulefill}\thinspace}
\providecommand{\MR}{\relax\ifhmode\unskip\space\fi MR }
\providecommand{\MRhref}[2]{%
  \href{http://www.ams.org/mathscinet-getitem?mr=#1}{#2}
}
\providecommand{\href}[2]{#2}

\mbox{ } \\

\end{document}